\documentclass[12pt,reqno]{amsart}

\usepackage[margin=1in]{geometry}
\usepackage[demo]{graphicx}
\usepackage{amsmath, amssymb, amsthm} 
\usepackage{lmodern}
\usepackage[T1]{fontenc}
\usepackage{fancyhdr}
\usepackage[dvipsnames]{xcolor} 
\usepackage{pdfcolmk}
\usepackage{tikz-cd}
\usepackage{mathtools}
\usepackage{mathrsfs}  
\usepackage{subfiles}
\usepackage{calligra}
\usepackage{ytableau}

\usepackage[maxbibnames=99,backend=biber, style=alphabetic]{biblatex}
\bibliography{Gorenstein}

\usepackage{newclude}

\usepackage{caption}
\usepackage{subcaption}

\usepackage{style}

\tikzset{
  closed/.style = {decoration = {markings, mark = at position 0.5 with { \node[transform shape, xscale = .8, yscale=.4] {/}; } }, postaction = {decorate} },
  open/.style = {decoration = {markings, mark = at position 0.5 with { \node[transform shape, scale = .7] {$\circ$}; } }, postaction = {decorate} }
}

\usepackage{hyperref}
\newcommand{\ns}{\text{ns}}
\newcommand{\ubal}{\text{ub}}

\title{When are splitting loci Gorenstein?}
\author{Feiyang Lin}

\begin{document}

\begin{abstract}
Splitting loci are certain natural closed substacks of the stack of vector bundles on $\bP^1$, which have found interesting applications in the Brill-Noether theory of $k$-gonal curves. In this paper, we completely characterize when splitting loci, as algebraic stacks, are Gorenstein or $\Q$-Gorenstein. The main ingredients of the proof are a computation of the class groups of splitting loci in certain affine extension spaces, and a formula for the class of their canonical module.
\end{abstract}

\maketitle
\tableofcontents

\section{Introduction}
Let $\cB_{r,d}$ be the stack of rank $r$, degree $d$ vector bundles on $\bP^1$. For every tuple of integers $\vec{e} = (e_1, \dots, e_r)$ with $e_1 \leq e_2 \leq \cdots \leq e_r$ and $e_1 + \cdots + e_r = d$, there is a locally closed substack $\SplittOpen \subseteq \cB_{r,d}$ parametrizing families of vector bundles on $\bP^1$ that fiberwise are isomorphic to $\cO(\vec{e})$. The closure of $\SplittOpen$ is a closed substack $\Splitt \subseteq \cB_{r,d}$, the \textit{splitting locus} associated to $\vec{e}$, parametrizing vector bundles on $\bP^1$ that split as $\cO(\vec{e}) = \cO(e_1)\oplus \cdots \oplus \cO(e_r)$ or a specialization of $\cO(\vec{e})$. While $\SplittOpen$ is smooth, the closure $\Splitt$ can be singular. In this paper, we completely characterize when splitting loci $\Splitt$ are Gorenstein or $\Q$-Gorenstein. Our universal statement extends to splitting loci in every scheme that maps smoothly to $\cB_{r,d}$ (i.e. if the Kodaira-Spencer map is surjective at every point). For example, our results exhibit the corresponding Hurwitz-Brill-Noether loci as having at most Gorenstein or $\Q$-Gorenstein singularities (see \cite{Larson2021} for more details). In previous work on the singularities of splitting loci, it has been shown that splitting loci are normal and Cohen-Macaulay \cite{LLV}, and that certain tame splitting loci have rational singularities \cite{Lin2025}.

We say that $\vec{e}$ is a \textit{block arithmetic progression} of block size $s$ and difference $t$ if it is a shift of the tuple $(0^s, t^s, \dots, ((m-1)t)^s)$ ($t^s$ is a short hand for $s$ copies of the number $t$). As usual, if $s = 1$, we say that the tuple is an \textit{arithmetic progression}. We say that $\vec{e}$ is \textit{contiguous} if 
$\vec{e}$ is a shift of a tuple of the form $(\dots, 0^a, 1^b, 2^a, 3^b, \dots)$ for some positive integers $a, b$.
\begin{example}
	\begin{enumerate}
		\item $\vec{e} = (-1,0,0,0,1)$ and $\vec{e} = (0,1,1,1,2,3,3,3)$ are contiguous.
		\item $\vec{e} = (-3,-3,0,0,3,3)$ is a block arithmetic progression of block size $2$ and difference $3$.
	\end{enumerate}
\end{example}

We say that a Cohen-Macaulay scheme $X$ is $N$-Gorenstein if $\omega_X^{\otimes N}$ is a line bundle. Being Gorenstein and $N$-Gorenstein are smooth-local properties (see Section \ref{sec:sm_local}), so it makes sense to ask when an algebraic stack, such as a splitting locus $\Splitt$, is Gorenstein or $N$-Gorenstein. We say that an algebraic stack is $\Q$-Gorenstein if it is $N$-Gorenstein for some $N \geq 1$. 

Our main theorem is the following.
\begin{theorem} \label{thm:main}
A splitting locus $\Splitt$ is Gorenstein if and only if $\vec{e}$ is a block arithmetic progression of difference $0, 1$ or $2$, or if $\vec{e}$ is contiguous. $\Splitt$ is $\Q$-Gorenstein (but not Gorenstein) if and only if it is an arithmetic progression of difference $t \geq 3$. In this case, the smallest $N$ for which $\Splitt$ is $N$-Gorenstein is $t$ if $t$ is odd, and $t/2$ if $t$ is even.
\end{theorem}

To prove Theorem \ref{thm:main}, we first exhibit a basis \[a_1^{(1)}, \dots, a_1^{(\ell)}, b_2^{(1)}, \dots, b_2^{(\ell)}\] for $A^1(\Splitt)$ in Proposition \ref{prop:basisForA1}. The $a_1^{(i)}$'s lift generators for $A^1(\SplittOpen)$. Setting $\vec{e} = (f_1^{s_1}, \dots, f_\ell^{s_\ell})$, let $\pvec{e}^{(i)}$ be defined by replacing $f_i^{s_i}$ in $\vec{e}$ with $(f_i - 1, f_i^{s_i - 2}, f_i+1)$. Then the $b_2^{(i)}$'s are given by 
\[
	b_2^{(i)} = 
	\begin{cases}
		[\mathbf{\overline{\Sigma}}_{\pvec{e}^{(i)}}] & \text{if } f_{i-1}+1 < f_i < f_{i+1}- 1 \text{ and } s_i \geq 2, \\ 
		0 & \text{otherwise}.
	\end{cases}
\]
Note that $\mathbf{\overline{\Sigma}}_{\pvec{e}^{(i)}}$ is of codimension $1$ in $\Splitt$. (When we say basis, we mean to include only those nonzero $b_2^{(i)}$'s.) In terms of these universal classes, we prove an adjunction-type formula for the class of the canonical module for splitting loci.

\begin{theorem}[Theorem \ref{prop:classCanonical}] \label{thm:classCanonical}
Let $X$ be a scheme, and let $\cE$ be a rank $r$, fiberwise degree $d$ vector bundle on $X \times \bP^1$, inducing a smooth map $X \to \cB_{r,d}$ (so that $X$ is necessarily smooth). Let $\vec{e} = (f_1^{s_1}, \dots, f_\ell^{s_\ell})$, where $f_1 < \cdots < f_\ell$, and let $\splitt(X)$ denote the splitting locus corresponding to $\vec{e}$ in $X$. Write $\delta_i = \sum_{j < i} s_j - \sum_{j > i} s_j$. Then
\[
	[\omega_{\splitt(X)}] = [\omega_X|_{\splitt(X)}] + \sum_{1 \leq i \leq \ell}\left ( (\deg \vec{e} - rf_i + \delta_i)a_1^{(i)} + (r-s_i)b_2^{(i)} \right),
\] 
where the $a_1^{(i)}, b_2^{(i)}$ classes denote the pullbacks of the universal classes on $\Splitt$ to $\splitt(X)$. 
\end{theorem}
Next, we consider a specific family of smooth covers of the stack $\vbst$ of rank $r$, degree $d$ vector bundles on $\bP^1$, namely the affine spaces $H^1 \cEnd (\cO(\pvec{e}'))$ where $\pvec{e}' = ((-M)^{r-1}, D)$, and $D$ is large enough so that $\pvec{e}' < \vec{e}$. Let $\splitt$ denote the splitting locus corresponding to the splitting type $\vec{e}$ inside $H^1 \cEnd (\cO(\pvec{e}'))$. As worked out in \cite[Section 2.2]{Lin2025}, for splitting types $\pvec{e}' = ((-M)^{r-1}, D)$, the affine space $H^1 \cEnd (\cO(\pvec{e}'))$ is a principal $GL_{r-1}$-bundle over an open substack of a vector bundle over $\vbst$. Using this interpretation of $H^1 \cEnd (\cO(\pvec{e}'))$, we prove the following: 

\begin{theorem} \label{thm:excision} Let $\pvec{e}' = ((-M)^{r-1}, D)$ and let $\vec{e} = (f_1^{s_1}, \dots, f_\ell^{s_\ell})$, where $f_1 < \cdots < f_\ell$, such that $\pvec{e}' < \vec{e}$. Then the class group of the splitting locus $\splitt \subseteq H^1 \cEnd (\cO(\pvec{e}'))$ fits in the following right exact sequence:
\begin{equation} \label{eq: excision}
	\Z^2 \to A^1(\Splitt) \to A^1(\splitt) \to 0,
\end{equation}
where the first map is given by sending the two basis elements to 
\begin{align*}
	\alpha_1 &= \sum_{i=1}^\ell \left (-(f_i+M)a_1^{(i)} + b_2^{(i)} \right ), \\ 
	\alpha_2 &= \sum_{i = 1}^\ell \left ((f_i+M+1)a_1^{(i)} -b_2^{(i)} \right ).
\end{align*}
\end{theorem}
With some computation, the previous two theorems together tell us exactly when the class of the canonical module in $H^1\cEnd(\cO(\pvec{e}'))$ is trivial or $N$-torsion. This turns out to be enough. The affine spaces $H^1\cEnd(\cO(\pvec{e}'))$ in general admit an action by $\bbG_m^{k-1}$, where $k$ is the number of perfectly balanced blocks in $\pvec{e}'$, which makes the ideal for splitting loci therein $\Z^{k-1}$-graded (Section \ref{sec:grading}). When $\pvec{e}'$ has the special shape of $((-M)^{r-1}, D)$, $H^1\cEnd(\cO(\pvec{e}'))$ admits a natural $\bbG_m$-action which makes the canonical module for $\splitt$ standard-graded. As such, $\omega_{\splitt}^{\otimes N}$ is a line bundle if and only if we have $N[\omega_{\splitt}] = 0$ as a class in $A^1(\splitt)$ (Proposition \ref{prop:graded}).
\begin{remark} 
\begin{enumerate}
	\item Note that the image of the first map in \eqref{eq: excision} is generated by the classes $\alpha_1 + \alpha_2 = \sum_{i = 1}^\ell a_1^{(i)}$ and $\alpha_2 - (M+1)(\alpha_1+\alpha_2) = \sum_{i = 1}^\ell (f_i a_1^{(i)} - b_2^{(i)})$, and in particular does not depend on $M$. One can check that these are precisely the classes that arise from restricting classes in $A^1(\cB_{r,d})$. A priori, these classes must pull back to zero on $\splitt$, since we can also factor the pullback through the affine space $H^1 \cEnd \cO(\pvec{e}')$, which has trivial class group. Our computation tells us that these are the only classes that become zero.
	\item When $\vec{e}$ is rank $2$, the ideals for splitting loci in $H^1 \cEnd\cO(\pvec{e}')$ are precisely Hankel determinantal ideals, as explained in \cite[Corollary 5.4]{Eisenbud-Schreyer}. Thus our results recover the computation in \cite{CONCA2018111} for the class groups of Hankel determinantal rings and the classes of the canonical modules therein. Our previous remark gives an explanation for why the class group of Hankel determinantal rings only depends on the difference between the number of rows and columns of the Hankel matrix. 
	\item Our investigation is partly inspired by the results in the case of determinantal varieties, warranting a brief comparison here. Over a field, the class group of the determinantal variety of $m$-by-$n$ matrices of rank at most $r > 0$ is $\Z$. It is freely generated by the class of one of two distinguished prime ideals $P, Q$, satisfying $[P] = -[Q]$, and the class of the canonical module is $m[P]+n[Q]$. In particular, the determinantal variety is Gorenstein (equivalently $\Q$-Gorenstein) precisely when the matrices are square, i.e. $m = n$ \cite[Theorem 8.8]{Bruns-Vetter}.
	\item Our work inches us one step closer to the question of whether $\Splitt$ has rational singularities when $\vec{e} = (-2,0,2)$, which is the first non-tame case. Since $\vec{e} = (-2,0,2)$ is Gorenstein, it now suffices to check that $\Splitt$ is log terminal. This seems out of reach to us, since our intersection-theoretic techniques are agnostic of any behavior in codimension $2$ and beyond, which one would need a better handle of in order to compute the discrepancies.
\end{enumerate}
\end{remark}

\noindent
{\textbf{Acknowledgements.}} 
Thank you to my advisors Hannah Larson and David Eisenbud for many helpful conversations. Thanks also to Martin Olsson for discussions about Proposition \ref{prop:normalBundle}. I made use of some Macaulay2 code written by Frank-Olaf Schreyer to observe that certain splitting types were Gorenstein. Part of this work was conducted while the author was supported by NSF Grant DMS 2001649. 

\section{Preliminaries}
We refer the reader to \cite{larson_degeneracy} and \cite[Section 2]{Lin2025} for background on splitting loci.
\subsection{A Torus Action On \texorpdfstring{$H^1 \cEnd(\cO(\pvec{e}))$}{Successive Extension Spaces}} 
\label{sec:grading}

Let $\vec{e} = (f_1^{s_1}, \dots, f_\ell^{s_\ell})$. We shall explain that there is a natural $\bbG_m^{\ell-1}$-action on $H^1 \cEnd(\cO(\vec{e}))$ under which splitting loci are invariant. In particular, the ideals for splitting loci in $H^1 \cEnd(\cO(\vec{e}))$ are naturally $\Z^{\ell-1}$-multigraded.

As explained in \cite[Section 2.2]{Lin2025}, a $T$-point of the affine space $H^1 \cEnd(\cO(\vec{e}))$ is the data of vector bundles $\cK_1, \dots, \cK_{\ell}$ on $T \times \bP^1$, plus the data of an isomorphism $\cK_1 \cong \cO(f_1^{s_1})|_{T\times \bP^1}$, plus the data of extensions
\[
	0 \to \cK_{i-1} \xrightarrow{f_i} \cK_{i} \xrightarrow{g_i} \cO(f_i^{s_i}) \to 0
\]
for each $2 \leq i \leq \ell$. Note that we will often write $\cO(f_i^{s_i})$ to denote the pullback of this bundle along the second projection from $\bP^1$ to $T \times \bP^1$. In terms of these extensions, the action is easy to describe. An element $(\alpha_1, \dots, \alpha_{\ell-1}) \in \bbG_m^{\ell-1}$ acts on such a collection of data by keeping the vector bundles the same, but scaling the extensions to:
\[
	0 \to \cK_{i-1} \xrightarrow{f_i} \cK_{i} \xrightarrow{\beta_i g_i} \cO(f_i^{s_i}) \to 0,
\]
where $\beta_{i} = \alpha_1 \alpha_2 \cdots \alpha_{i-1}$. After fixing an isomorphism of $\cO(\vec{e}) \cong \oplus_{i = 1}^\ell \cO(f_i^{s_i})$, we can identify $H^1 \cEnd(\cO(\vec{e})) \cong \oplus_{1 \leq j < i \leq \ell} \Ext^1(\cO(f_i^{s_i}), \cO(f_j^{s_j}))$. We will now work out what the action is in these coordinates by considering the action on a $k$-point. 

\begin{proposition} The isomorphism $H^1 \cEnd(\cO(\vec{e})) \cong \oplus_{1 \leq j < i \leq \ell} \Ext^1(\cO(f_i^{s_i}), \cO(f_j^{s_j}))$ is a decomposition of $H^1 \cEnd(\cO(\vec{e})$ into weight spaces of the action of $\bG_m^{\ell-1}$. 
For $1 \leq j < i \leq \ell$, if $v_{ij} \in \Ext^1(\cO(f_i^{s_i}), \cO(f_j^{s_j}))$, and $\alpha = (\alpha_1, \dots, \alpha_{\ell-1}) \in \bG_m^{\ell-1}$, then $\alpha \cdot v_{ij} = \beta_j^{-1}\beta_i v_{ij} = \alpha_j \alpha_{j+1} \cdots \alpha_{i-1} v_{ij}$.
\end{proposition}

\begin{proof}
Let $K_1, \dots, K_\ell$ be the bundles involved in the data of a $k$-point of $H^1 \cEnd(\cO(\vec{e}))$, corresponding to the point 
$(v_{ij})_{1 \leq j < i \leq \ell} 
\in 
\oplus_{1 \leq j < i \leq \ell} 
\Ext^1(\cO(f_i^{s_i}), \cO(f_j^{s_j}))$. We will show that $\alpha = (\alpha_1, \dots, \alpha_{\ell-1}) \in \bG_m^{\ell-1}$ acts by $v_{ij} \mapsto \beta_j^{-1}\beta_i v_{ij}$.

To prove this, we proceed by induction. Suppose that we have shown that $\alpha \cdot v_{i'j} = \beta_{j}^{-1}\beta_{i'} v_{i'j}$ whenever $i' < i$.
By definition, $\alpha$ acts on $\Ext^1(\cO(f_i^{s_i}), K_{i-1})$ by scaling the extensions with weight $\beta_i$ for all $i \geq 2$. The identification of $\Ext^1(\cO(f_i^{s_i}), K_{i-1})$ with $\oplus_{1 \leq j < i} \Ext^1(\cO(f_i^{s_i}), \cO(f_j^{s_j}))$ can be made step by step. Suppose we have identified 
\[\Ext^1(\cO(f_i^{s_i}), K_{i-1})
\cong
\Ext^1(\cO(f_i^{s_i}), K_{t-1}) \oplus 
\oplus_{t \leq j < i} \Ext^1(\cO(f_i^{s_i}), \cO(f_j^{s_j})),\] 
where $t \geq 3$. Then in order to identify $\Ext^1(\cO(f_i^{s_i}), K_{t-1})$ with $\Ext^1(\cO(f_i^{s_i}), \cO(f_{t-1}^{s_{t-1}})) \oplus \Ext^1(\cO(f_i^{s_i}), K_{t-2})$, we choose a (non-canonical) splittings of the following exact sequence:
\begin{equation} \label{eq:SES_extensions}
	0 \to \Ext^1(\cO(f_i^{s_i}), K_{t-2}) \to 
	\Ext^1(\cO(f_i^{s_i}), K_{t-1}) \to 
	\Ext^1(\cO(f_i^{s_i}), \cO(f_{t-1}^{s_{t-1}})) \to 0,
\end{equation}
Suppose by induction that $\alpha$ acts on $\Ext^1(\cO(f_i^{s_i}), K_{t-1})$ by scaling by $\beta_i$, and that we have also shown the desired claim that $\alpha \cdot v_{ij} = \beta_j^{-1}\beta_i v_{ij}$ whenever $j \geq t$. The base case when $t = i$ is certainly true. We will show that $\alpha$ acts on $\Ext^1(\cO(f_i^{s_i}), K_{t-2})$ by $\beta_i$ as well, and that $\alpha \cdot v_{i, t-1} = \beta_{t-1}^{-1}\beta_i v_{i,t-1}$.

Let us first consider the second map in \eqref{eq:SES_extensions} above. The map is given by taking the pushout along the given map $K_{t-1} \xrightarrow{g_{t-1}} \cO(f_{t-1}^{s_{t-1}})$:
\[
	\begin{tikzcd}
0 \arrow[r] 
& K_{t-1} \arrow[d,"g_{t-1}"] \arrow[r, "f"]           
& F_{t-1, i} \arrow[r, "g"] \arrow[d, "\psi"] 
& \cO(f_i^{s_i}) \arrow[r] \arrow[d,-, double equal sign distance] & 0 \\
0 \arrow[r] & \cO(f_{t-1}^{s_{t-1}}) \arrow[r, "\phi"] & E_{t-1,i} \arrow[r, "\gamma"]           & \cO(f_i^{s_i}) \arrow[r]                                & 0
\end{tikzcd}.
\]

Since $\alpha$ acts on the first extension by sending $g$ to $\beta_i g$ by our inductive hypothesis, the prescribed action on the original extensions would act on $\Ext^1(\cO(f_i^{s_i}), \cO(f_{i-1}^{s_{i-1}}))$ by scaling by $\beta_{t-1}^{-1}\beta_i$, as explained by the following diagram.
\[
	\begin{tikzcd}
0 \arrow[r] & K_{t-1} \arrow[d,"\beta_{t-1} g_{t-1}"] \arrow[r, "f"]           & F_{t-1, i} \arrow[r, "\beta_i g"] \arrow[d, "\beta_{t-1} \psi"] & \cO(f_i^{s_i}) \arrow[r] \arrow[d,-, double equal sign distance] & 0 \\
0 \arrow[r] & \cO(f_{t-1}^{s_{t-1}}) \arrow[r, "\phi"] & E_{t-1, i} \arrow[r, "\beta_{t-1}^{-1}\beta_i \gamma"]           & \cO(f_i^{s_i}) \arrow[r]                                & 0
\end{tikzcd}
\]

We can also compute the restriction of the action to $\Ext^1(\cO(f_i^{s_i}), K_{t-2})$, by interpreting the first map in \eqref{eq:SES_extensions}. It is again given by taking a pushout along the map $g_{t-1}: K_{t-2} \to K_{t-1}$. So the action of $\alpha$ on $\Ext^1(\cO(f_i^{s_i}), K_{t-1})$ with weight $\beta_i$ restricts as the following diagram shows:
\[
\begin{tikzcd}
0 \arrow[r] 
& K_{t-2} \arrow[r, "\phi'"] \arrow[d,"f_{t-1}"] 
& F_{t-2, i} \arrow[r, "\beta_i \gamma'"]  \arrow[d, "\psi' "]        
& \cO(f_i^{s_i}) \arrow[r] \arrow[d,-, double equal sign distance] 
& 0 \\ 
0 \arrow[r]
& K_{t-1} \arrow[r, "f"]           
& F_{t-1,i} \arrow[r, "\beta_i g"] 
& \cO(f_i^{s_i}) \arrow[r] 
& 0 
\end{tikzcd}.
\]
Thus $(\alpha_1, \dots, \alpha_{\ell-1}) \in \bbG_m^{\ell-1}$ acts on $\Ext^1(\cO(f_i^{s_i}), K_{t-2})$ also by scaling by $\beta_i$.
\end{proof}

Since the action above arises from scaling extension classes, it does not change the isomorphism class of the bundles $K_1, \dots, K_{\ell}$. In particular, splitting loci corresponding to the isomorphism class of $K_{\ell}$ are invariant under this action.

\begin{example}
Let $\vec{e} = (-2,0,2)$. Then we can write the coordinate ring of $H^1\cEnd(\cO(\vec{e}))$ as $k[a, b, c_1, c_2, c_3]$, where $a, b$ are coordinates corresponding to the pairs $(-2, 0)$ and $(0, 2)$ respectively, and $c_1, c_2, c_3$ are coordinates of $\Ext^1(\cO(2), \cO(-2))$. Then there is a natural $\bbG_m^2$-action on this ring, giving rise to a $\Z^2$-grading. Namely, $\deg a = (1,0), \deg b = (0,1)$ and $\deg c_i = (1,1)$. All splitting loci inside $H^1\cEnd(\cO(\vec{e}))$ are homogeneous with respect to this grading.
\end{example}

\begin{example} Of the most importance to us is the easiest case where $\vec{e} = (f_1^{s_1}, f_2^{s_2})$. In this case, $H^1 \cEnd(\cO(\vec{e})) = \Ext^1(\cO(f_2^{s_2}), \cO(f_1^{s_1}))$ admits a $\bbG_m$-action that corresponds to scaling the extension class. The corresponding $\Z$-grading on the polynomial ring with $s_1s_2(f_2-f_1-1)$ variables is the standard grading, and ideals for splitting loci are homogeneous with respect to this grading.
\end{example}

\subsection{Gorenstein and \texorpdfstring{$N$}{N}-Gorenstein Algebraic Stacks} \label{sec:sm_local}
Let us check that both being Gorenstein and being $N$-Gorenstein are smooth-local properties for locally Noetherian, Cohen-Macaulay schemes.
\begin{lemma}
Let $X$ be a locally Noetherian Cohen-Macaulay scheme with a smooth cover by schemes $U_{\alpha}$ indexed by $\alpha \in I$. Then $X$ is ($N$-)Gorenstein if and only if $U_{\alpha}$ is ($N$-)Gorenstein for every $\alpha$.
\end{lemma}
\begin{proof}
Since a scheme is Gorenstein if and only if every local ring is Gorenstein, and since being locally Noetherian and Cohen-Macaulay are both smooth-local properties, we can reduce to the following claim: let $A, B$ be Noetherian, Cohen-Macaulay local rings and let $f: A \to B$ be a smooth and local homomorphism. Then $A$ is ($N$-)Gorenstein if and only if $B$ is ($N$-)Gorenstein.

Since $f$ is smooth, there exists a relative dualizing module $\omega_f$ which is a rank one projective $B$-module. Since $B$ is local, in fact $\omega_f \cong B$. By \cite[Lemma 0E30]{stacks-project}, $\omega_B = \omega_f \otimes_B (B \otimes_A \omega_A)$ (the derived tensor product is just the regular tensor product because $\omega_f$ is projective, and the derived pullback of $\omega_A$ to $\Spec B$ is just the regular pullback because $f$ is smooth and in particular flat). So in fact $\omega_B = \omega_A \otimes_A B$, and $\omega_B^{\otimes N} = \omega_A^{\otimes N} \otimes_A B$. Since $f: A \to B$ is smooth and in particular faithfully flat, given a finitely generated $A$-module M, $M \otimes_A B$ is free iff $M \otimes_A B$ is projective iff $M \otimes_A B$ is flat iff $M$ is flat iff $M$ is projective iff $M$ is free. Therefore $\omega_A^{\otimes N} \cong A$ if and only if $\omega_B^{\otimes N} \cong B$.

Alternatively just for the Gorenstein claim, one could use \cite[Lemma 0BJL]{stacks-project} ($B/\fm_B$ is Gorenstein because it's the fiber of a smooth morphism, so in particular should be regular).
\end{proof}

Thus as with all smooth-local properties, if $\fX$ is a locally Noetherian, Cohen-Macaulay algebraic stack, it makes sense to talk about whether $\fX$ is Gorenstein or $N$-Gorenstein. We say that $\fX$ is $\Q$-Gorenstein if there exists $N \geq 1$ such that $\fX$ is $N$-Gorenstein.

\subsection{Reduction to a class calculation}
In this subsection, we explain how to reduce the question of whether splitting loci are Gorenstein or $\Q$-Gorenstein to a class calculation. 

\begin{proposition} Let $S$ be a $\Z$-graded Noetherian ring with $S_0 \cong k$. Let $M$ be a graded finitely generated $S$-module. Then $M$ is locally free if and only if it is a free $S$-module.
\end{proposition}
\begin{proof}
We shall show the forward direction. Suppose that $M$ is a graded, finitely generated, locally free $S$-module. In particular, $M$ is projective. By the graded Nakayama's lemma, there exists a graded surjection from a graded free $S$-module which is an isomorphism after tensoring with $k$. Let the kernel be $K$, so that we have a short exact sequence 
\[
	0 \to K \to \oplus S(-d_i) \twoheadrightarrow M \to 0
\]
Then the fact that $M$ is projective implies that tensoring with $k$ is exact, which implies that $K \otimes_S k = 0$. But by the graded Nakayama's lemma, this implies that $K = 0$. In particular, $M$ is free.  
\end{proof}

Recall that on an integral and normal scheme, the Weil class group can also be viewed as the group of rank $1$ reflexive sheaves (for details see e.g. \cite{SchwedeReflexive}). As such, given a rank $1$ reflexive sheaf $M$ on an integral and normal scheme $X$, we use $[M] \in A^1(X)$ to denote the corresponding class. Recall from the introduction that $\splitt$ denotes the splitting locus inside $H^1\cEnd(\cO(\pvec{e}'))$, where $\pvec{e}' = ((-M)^{r-1}, D)$ and $\pvec{e}' > \vec{e}$. We will write $\omega_{\splitt}$ for the canonical module of $\splitt$.

\begin{proposition} \label{prop:graded} $\omega_{\splitt}^{\otimes N}$ is locally free if and only if $N[\omega_{\splitt}] \in A^1(\overline{\Sigma}_{\vec{e}})$ is zero.
\end{proposition}
\begin{proof}
The canonical module $\omega_{\splitt}$ is a rank $1$ reflexive sheaf (e.g. \cite[Theorem 3.3.10(d), Proposition 3.3.13]{Bruns_Herzog_1998}). Therefore $N[\omega_{\splitt}]$ is zero if and only if $\omega_{\splitt}^{\otimes N}$ is the structure sheaf. But since $\omega_{\splitt}^{\otimes N}$ is graded, by the previous proposition, it is the structure sheaf if and only if it is locally free.
\end{proof}

\section{Some class calculations}

\subsection{Codimension One Chow Group of \texorpdfstring{$\Splitt$}{Universal Splitting Loci}} \label{sub:codimension_one_chow_group_of_splitt}
Let $\vec{e} = (f_1^{s_1}, \dots, f_\ell^{s_\ell})$, where $f_1 < \cdots < f_\ell$. Let $r = \sum_{i = 1}^\ell s_i = \rk \vec{e}$, and let $d = \sum_{i=1}^\ell s_i f_i = \deg \vec{e}$. Recall from the proof of \cite[Proposition 9.1]{LLV} that if $\pvec{e}'$ is such that $u(\pvec{e}')-u(\vec{e}) = 1$, then there exists $1 \leq i \leq \ell$ such that $f_{i-1} + 1 < f_i < f_{i+1}-1$ (as far as $i$ has neighbors), $s_i \geq 2$, and $\pvec{e}'$ arises from replacing $f_i^{s_i}$ in $\vec{e}$ with $(f_i - 1, f_i^{s_i - 2}, f_i+1)$. We denote by $\pvec{e}^{(i)}$ the codimension one splitting type $\pvec{e}'$ obtained by unbalancing $f_i^{s_i}$ when it exists. Let $\bZ_{\vec{e}}$ be the union of splitting loci $\SplittOpen$ and all $\mathbf{\Sigma}_{\pvec{e}'}$ where $u(\pvec{e}')-u(\vec{e}) = 1$. 

Let $\SplittNye$ be the algebraic stack whose functor is given by 
\begin{align*}
  \left (T \to \Spec k \right ) 
  & \mapsto 
  \left \{
    \begin{tabular}{lll}
    \text{A vector bundle $\cE$ of rank $r$ and fiberwise degree $d$  and a flag of} \\
    \text{quotients $\cE = \cQ_0 \twoheadrightarrow \cQ_{1}
    \twoheadrightarrow \cQ_{2} 
    \twoheadrightarrow \cdots 
    \twoheadrightarrow \cQ_{\ell-1}$ on $T \times \bP^1$, such that $\cQ_{i}$ is} \\ 
    \text{flat over $T$, and fiberwise of rank $\sum_{j=1}^{\ell-i} s_j$ and degree $\sum_{j=1}^{\ell-i} s_jf_j$}
    \end{tabular}
  \right\} / \sim,
\end{align*}
where two flags of quotients $\cE = \cQ_0 \twoheadrightarrow \cQ_{1}
    \twoheadrightarrow \cQ_{2} 
    \twoheadrightarrow \cdots 
    \twoheadrightarrow \cQ_{\ell-1}$, $\cE = \cQ_0 \twoheadrightarrow \cQ_{1}'
    \twoheadrightarrow \cQ_{2}' 
    \twoheadrightarrow \cdots 
    \twoheadrightarrow \cQ_{\ell-1}'$ are equivalent if there are isomorphisms $\cQ_i \to \cQ_i'$ which together with the identity morphism on $\cQ_0$ form a commutative diagram.
By \cite[Theorem 5]{Lin2025}, $\SplittNye$ is a resolution of singularities for $\Splitt$. Moreover, for each $\pvec{e}^{(i)}$, the only possible such flags are coarsenings of the Harder-Narasimhan flag for $\cO(\pvec{e}^{(i)})$, so the natural map $\SplittNye \to \Splitt$ is not only an isomorphism over $\SplittOpen$, but also an isomorphism over $\bZ_{\vec{e}}$. As such, on $\bZ_{\vec{e}} \times \bP^1$, there is a natural flag of quotients for the universal bundle
\[
	\cE = \cQ_0 \twoheadrightarrow \cQ_{1}
    \twoheadrightarrow \cQ_{2} 
    \twoheadrightarrow \cdots 
    \twoheadrightarrow \cQ_{\ell-1},
\]
where each $\cQ_i$ is in fact a vector bundle. For $1 \leq i < \ell$, let $\cN_{\ell-i+1} = \ker(\cQ_{i-1} \to \cQ_{i})$, and let $\cN_{1} = \cQ_{\ell-1}$. The sheaves $\cN_i$ are vector bundles on $\bZ_{\vec{e}} \times \bP^1$, which split as $\cO(f_i^{s_i})$ over $\bZ_{\vec{e}} \setminus \mathbf{\Sigma}_{\pvec{e}^{(i)}}$, and split as $\cO(f_i-1, f_i^{s_i-2}, f_{i}+1)$ over $\mathbf{\Sigma}_{\pvec{e}^{(i)}}$, if it exists. 

Let $z$ be the hyperplane class on $\bZ_{\vec{e}} \times \bP^1$, corresponding to $\cO(1)$. For all $k \geq 1$, there exist classes $a_k^{(i)} \in A^k(\bZ_{\vec{e}}), b_k^{(i)} \in A^{k-1}(\bZ_{\vec{e}})$ such that
\[
	c_k(\cN_i(-f_i)) = \pi^* a_k^{(i)} + \pi^* {b_k^{(i)}} z.
\]  
If $\pvec{e}^{(i)}$ does not exist, then 
\begin{equation} \label{eq:NiIsPulledBack}
\cN_i(-f_i) \cong \pi^* \pi_* (\cN_i(-f_i)),
\end{equation} 
and in particular $b_k^{(i)} = 0$ for all $k \geq 1$. If ${\pvec{e}^{(i)}}$ exists, then away from $\mathbf{\Sigma}_{\pvec{e}^{(i)}}$, we have \eqref{eq:NiIsPulledBack} again, which shows that up to classes in the image of $A^*(\mathbf{\Sigma}_{\pvec{e}^{(i)}})$, $b_k^{(i)} \equiv 0$. Note also that $b_1^{(i)}$ is the degree of $\cN_i(-f_i)$ on fibers, which is zero.

\begin{lemma} \label{lem:GRRCalc} For all $M \in \Z$,
\begin{align*}
	c_1 (R\pi_* \cN_i(-f_i+M)) &= c_1(\pi_* \cN_i(-f_i)) + M\pi_*(c_1(\cN_i) . z ) \\ 
	&= (M+1)a_1^{(i)} - b_2^{(i)}.
\end{align*}
\end{lemma}
\begin{proof}
Twisting changes the Chern classes as follows:
\begin{align*}
	c_1(\cN_i(-f_i+M)) &= \pi^* a_1^{(i)} + Ms_i z, \\ 
	c_2(\cN_i(-f_i+M)) &= \pi^* a_2^{(i)} + \pi^* b_2^{(i)} z + (s_i - 1) M\pi^* a_1^{(i)} z.
\end{align*}
Recall that the relative Todd class of a trivial $\bP^1$-bundle is $1+z$. So by Grothendieck-Riemann-Roch, 
\begin{align*}
c_1 (R\pi_* \cN_i(-f_i+M)) &= \pi_*
\left (
\frac{1}{2}(c_1(\cN_i(-f_i+M))^2 - 2c_2(\cN_i(-f_i+M))) + c_1(\cN_i(-f_i+M)).z
\right ) \\ 
&=
\pi_*
\left (
\frac{1}{2}(2M s_i \pi^* a_1^{(i)} . z - 2\pi^* b_2^{(i)} . z - 2M(s_i-1) \pi^*a_1^{(i)} . z ) + \pi^* a_1^{(i)} . z
\right )  \\ 
& = (M+1)a_1^{(i)} - b_2^{(i)}.
\end{align*}
\end{proof}

\begin{proposition} \label{prop:basisForA1}
Whenever the codimension one splitting type ${\pvec{e}^{(i)}}$ exists, $[\mathbf{\Sigma}_{\pvec{e}^{(i)}}] = b_2^{(i)} \in A^1(\bZ_{\vec{e}})$. In particular, $A^1(\Splitt)$ has a basis given by $a_1^{(1)}, \dots, a_1^{(\ell)}$ and those $b_2^{(i)}$'s for which $\pvec{e}^{(i)}$ exists.
\end{proposition}
\begin{proof}
By Lemma \ref{lem:GRRCalc},
\begin{align*}
	c_1(\pi_* \cN_i(-f_i)) &= a_1^{(i)} - b_2^{(i)}, \\ 
	c_1(\pi_* \cN_i(-f_i+1)) &= 2a_1^{(i)} - b_2^{(i)}.
\end{align*}
Moreover, as explained in the proof of \cite[Proposition 9.1]{LLV}, $\mathbf{\Sigma}_{\pvec{e}^{(i)}}$ is Cartier in $\bZ_{\vec{e}}$ and cut out by a section of $\wedge^{2s_i} \pi_*(\cN_i(-f_i) \otimes H^0(\bP^1, \cO(1)))^{\vee} \otimes \wedge^{2s_i} \pi_* \cN_i(-f_i+1)$. In particular, 
\[
	[\mathbf{\Sigma}_{\pvec{e}^{(i)}}] = c_1(\pi_* \cN_i(-f_i+1)) - 2c_1(\pi_* \cN_i(-f_i)) = b_2^{(i)}
\]
as desired.

By Lemma 3.1 of \cite{CanningLarson}, we know that $A^1(\mathbf{\Sigma}_{\vec{e}})$ is freely generated by $c_1(\pi_*\cN_i(-f_i))|_{\SplittOpen}$.
By excision, it follows that $A^1(\Splitt)$ is generated by
lifts of $c_1(\pi_*\cN_i(-f_i))|_{\SplittOpen}$ from $\SplittOpen$ to $\Splitt$ and the classes of these codimension one strata. The $a_1^{(i)}$ classes certainly lift $c_1(\pi_*\cN_i(-f_i))$, and we just saw that when $\pvec{e}^{(i)}$ exists, then $b_2^{(i)}$ is exactly the class of that stratum. Thus it remains to show that there are no relations between the $b_2^{(i)}$'s. Consider the map $\cB_{s_i, 0} \to \Splitt$ induced by the vector bundle $\cN(f_i) \oplus \oplus_{j \neq i} \cO(f_j^{s_j})$, where $\cN$ is the universal vector bundle on $\cB_{s_i, 0} \times \bP^1$. Under this map, the classes $a_1^{(j)}, b_2^{(j)}$ all pull back to zero for $j \neq i$, and $a_1^{(i)}, b_2^{(i)}$ pull back to the classes $a_1, a_2'$ as defined in \cite{Larson2023}, which are free generators of the (rational) Picard group of $\cB_{s_i, 0}$. In particular, if there were any relation $\sum_{j = 1}^\ell (c_j b_2^{(j)} + d_j a_1^{(j)}) = 0$ in $A^1(\Splitt)$, then it must be the case that $c_i a_2' + d_i a_1 = 0$ in the rational Chow ring of $\cB_{s_i, 0}$ for all $1 \leq i \leq \ell$, but that implies that $c_i = d_i = 0$. So in fact there cannot be any relations and the $a_1^{(i)}$ and $b_2^{(i)}$'s form a basis.
\end{proof}

\subsection{Proof of Theorem \ref{thm:excision}}
Let $\pvec{e}' = ((-M)^{r-1}, D)$, such that $D - (r-1)M = d$, and $D$ is large enough so that $\pvec{e}' < \vec{e}$. Let $\cE$ be the universal bundle on $\vbst \times \bP^1$. We will also use $\cE$ to denote the pullback of the universal bundle on $\vbst \times \bP^1$ to various spaces $U \times \bP^1$ when there is an obvious map $U \to \vbst$. 
When there is an obvious map $U \to \Splitt$, we also omit the pullback from our notation, writing $a_1^{(i)}, b_2^{(i)}$ to denote the pullbacks of the universal classes on $\Splitt$.
We will write $U_{\pvec{e}'} = H^1 \cEnd(\cO(\pvec{e}')) = \Ext^1(O(D), \cO(-M)^{\oplus r-1})$. Let us recall from Section 2.2 of \cite{Lin2025} how to construct the affine space $U_{\pvec{e}'}$ from $\vbst$: 
\begin{enumerate} 
	\item First we restrict to the open $U_{\pvec{e}'}^{(2)}$ in $\vbst$ where $\cE$ splits as $\pvec{e}'' \geq \pvec{e}'$.
	\item Over this open, we form the pushforward $V = \pi_* \cHom(\cE, \cO(D))$, which is a vector bundle.
	\item We excise the locus where the corresponding map $\phi: \cO(\pvec{e}'') \to \cO(D)$ is not surjective, and the locus where the kernel is not perfectly balanced, and obtain an open locus $U^{(1)}_{\pvec{e}'}$. 
	\item On $U_{\pvec{e}'}^{(1)}  \times \bP^1$, there is a universal surjection $\cE \to \cO(D)$ so that fiberwise the kernel is perfectly balanced. Let $\cK = \ker(\cE \to \cO(D))$. The last step now is to form the principal $GL_{r-1}$-bundle $U_{\pvec{e}'} = U^{(0)}_{\pvec{e}'} = \Isom(\cO(-M)^{r-1},\cK)$ which frames $\cK$. 
\end{enumerate}

\begin{figure}
\[
\begin{tikzcd}
U_{\pvec{e}'}^{(0)} \arrow[d, "GL_{r-1}\text{-bundle}"] & & \\ 
U_{\pvec{e}'}^{(1)}  \arrow[r, hook, open] & {V = \pi_* \cHom(\cE, \cO(D))} \arrow[d, "\text{vector bundle}"] &       \\ 
& U_{\pvec{e}'}^{(2)} = \bigcup_{\pvec{e}'' \geq \pvec{e}'} \mathbf{\Sigma}_{\pvec{e}''} \arrow[r, hook, open] & \vbst
\end{tikzcd}
\]
\caption{Building $\Ext^1(O(D), \cO(-M)^{\oplus r-1})$ from $\vbst$}
\end{figure}

After base change from $U_{\pvec{e}'}^{(2)}$ to $\Splitt^{\circ} = \cup_{\pvec{e}' \leq \vec{f} \leq \vec{e}} \mathbf{\Sigma}_{\vec{f}}$, we note that this construction is also a recipe for building $\splitt$ from $\Splitt^{\circ}$. 

Luckily, this construction is pretty amenable to tracking how $A^1(\splitt)$ compares to $A^1(\Splitt)$. Forming a vector bundle does not affect Chow, and as far as $A^1$ is concerned, forming a $GL$-bundle amounts to setting $c_1$ of the corresponding vector bundle to zero. We have also excised two closed loci in this process, and we must take note of whether they contain codimension one components, and if so, we must remember to excise them in the class group computation. The rest of this subsection will follow this recipe to prove Theorem \ref{thm:excision}.

Let $Z_{\ns}$ denote the closed locus in $V = \pi_* \cHom(\cE, \cO(D))$ of those maps $\cO(\vec{f}) \to \cO(D)$ which are not surjective. We label the relevant maps as follows:\[
	\begin{tikzcd}
V \times \bP^1 \arrow[r, "\tilde{p}"] \arrow[d, "\pi'"] & \Splitt^{\circ} \times \bP^1 \arrow[d, "\pi"] \\
V \arrow[r, "p"]                                        & \Splitt^{\circ}                              
\end{tikzcd}
\]

\begin{lemma} 
The closed locus $Z_{\ns}$ is irreducible of codimension $r-1$. In particular, $Z_{\ns}$ is of codimension $1$ if and only if $r = 2$. In that case,
\[
[Z_{\ns}] = \pi'_* c_2(\cHom(\cE, \cO(D))) = 
\begin{cases}
-(f_1+ M)a_1^{(1)} - (f_2 + M)a_1^{(2)} & \text{if } \vec{e} = (f_1, f_2), f_1 < f_2; \\ 
-(f+ M)a_1 + b_2 & \text{if }\vec{e} = (f, f).
\end{cases}
\]
Note that in the latter case, since $m = 1$, we use $a_1, b_2$ as a shorthand for $a_1^{(1)}, b_2^{(1)}$.
\end{lemma}
\begin{proof}
Let $Z'$ be the locus of points in $V \times \bP^1$ where the universal map $\cE \to \cO(D)$ is zero. By our choice of $D$, the evaluation map 
\[
	\operatorname{ev}: \pi^* V = \pi^* \pi_* \cHom(\cE, \cO(D)) \to \cHom(\cE, \cO(D))
\]
is surjective, and $Z'$ is in fact the total space of the vector bundle $W = \ker (\operatorname{ev})$ on $\mathbf{\overline{\Sigma}}_{\vec{e}}^{\circ} \times \bP^1$. In particular $Z'$ must be irreducible of codimension $r = \rk \cHom(\cE, \cO(D))$ in $V \times \bP^1$. A map $\cO(\vec{f}) \to \cO(D)$ is not surjective if it is zero at some point of $\bP^1$. Therefore, the nonsurjective locus $Z_{\ns} \subset V$ is the image of $Z'$ in $V$. Since a generic nonsurjective map $\cO(\vec{f}) \to \cO(D)$ drops rank at exactly one point, $[Z_{\ns}] = \pi_*' c_{r}(\cHom(\cE, \cO(D)))$. As this class has codimension $r-1$, it has codimension $1$ if and only if $r = 2$. 

Now suppose that $\vec{e} = (f_1, f_2)$ where $f_1 < f_2$. Then $\bZ_{\vec{e}} = \SplittOpen$ and we have the short exact sequence over $\bZ_{\vec{e}}$
\[
	0 \to \cN_2 \to \cE \to \cN_1 \to 0.
\]
Note that in this case, $D = f_1 + f_2 + M$. By Lemma \ref{lem:GRRCalc}, 
\begin{align*}
	\pi_*' c_2(\cHom(\cE, \cO(D))) &= \pi_*' \big ( (c_1(\cN_1^{\vee}(f_1)) + (D-f_1)z)(c_1(\cN_2^{\vee}(f_2)) + (D-f_2)z) \big ) \\ 
	&= \pi_*' \left ( (-\pi^* a_1^{(1)} + (f_2+M)z)(-\pi^* a_1^{(2)} + (f_1+M)z) \right ) \\ 
	&= -(f_1+M)a_1^{(1)} - (f_2+M)a_1^{(2)}.
\end{align*}
Lastly, suppose that $\vec{e} = (f, f)$. Then $m = 1$, $\cE = \cN_1$, $M = D-2f$, and we get  
\begin{align*}
	\pi_*' c_2(\cHom(\cE, \cO(D))) &= \pi_*' \big ( c_2(\cE^{\vee}(f)) + c_1(\cE^{\vee}(f))(D-f)z \big ) \\ 
	&= \pi_*' \left ( (\pi^* a_2 + \pi^* b_1 . z) -\pi^* a_1 . (f+M)z) \right ) \\
	&= -(f+M)a_1 + b_1.
\end{align*}
\end{proof}

Let $Z_{\ubal}$ denote the closed locus in $\pi_* \cHom(\cE, \cO(D)) \setminus Z_{\ns}$ of those surjective maps $\cO(\vec{f}) \to \cO(D)$ whose kernel is not isomorphic to $\cO(-M)^{\oplus (r-1)}$. Note that $Z_{\ubal}$ is empty if $r = 2$.

\begin{lemma} Suppose that $r \geq 3$. Then $Z_{\ubal}$ is of pure codimension $1$ and 
\[
	[Z_{\ubal}] = \sum_{i = 1}^\ell ((-f_i-M)a_1^{(i)} + b_2^{(i)}) \in A^1(V \setminus Z_{\ns}).
\]
\end{lemma}
\begin{proof}
Since we are always working with a finite type stack, there exists $S$ sufficiently large such that $\cK(S)$ is globally generated. The Str\o mme sequence \cite[Proposition 1.1]{Stromme1987} tells us that
\[
	0 \to \pi^* (\pi_* \cK(S-1))(-1) \to \pi^* \pi_* \cK(S) \to \cK(S) \to 0.
\]
Twisting this down by $\cO(-S+M-1)$ and pushing forward, we see that
\[
	\pi_* \cK(S-1) \otimes H^1 \cO_{\bP^1}(-S+M-2) \to \pi_* \cK(S) \otimes H^1 \cO_{\bP^1}(-S+M-1) \to R^1 \pi_* \cK(M-1) \to 0
\]
is exact. The locus where $\cK$ is not perfectly balanced is precisely the support of $R^1 \pi_* \cK(M-1)$.

In the resolution above of $R^1 \pi_* \cK(M-1)$ by a map of vector bundles, the two vector bundles involved both have rank $(r-1)(S-M)(S-M+1)$. So the support of $R^1 \pi_* \cK(M-1)$ is cut out by the section of a line bundle. Since it is possible for $\cK$ to split as perfectly balanced, this section cannot be zero, and since we are working with an integral stack, this shows that $Z_{\ubal}$ is of pure codimension one. The class of $Z_{\ubal}$ is just $(S-M)c_1(\pi_* \cK(S)) - (S-M+1)c_1(\pi_* \cK(S-1))$. 

Now, recall that on $(V \setminus Z_{\ns}) \times \bP^1$,
\[
	0 \to \cK \to \cE \to \cO(D) \to 0,
\]
and as such we have $c_1(\pi_* \cK(S)) = c_1(\pi_* \cE(S))$ and $c_1(\pi_* \cK(S-1)) = c_1(\pi_* \cE(S-1))$. By Lemma \ref{lem:GRRCalc}, 
\begin{align*}
[Z_{\ubal}] & = (S-M)c_1(\pi_* \cK(S)) - (S-M+1)c_1(\pi_* \cK(S-1)) \\ 
&= (S-M)c_1(\pi_* \cE(S)) - (S-M+1)c_1(\pi_* \cE(S-1)) \\ 
&= \sum_{i = 1}^\ell \left ( (S-M)c_1(\pi_* \cN_i(S)) - (S-M+1)c_1(\pi_* \cN_i(S-1)) \right ) \\ 
&= \sum_{i = 1}^\ell \left ( (S-M)((f_i+S+1)a_1^{(i)}-b_2^{(i)}) - (S-M+1)((f_i+S)a_1^{(i)}-b_2^{(i)}) \right ) \\ 
&= \sum_{i = 1}^\ell ((-f_i-M)a_1^{(i)} + b_2^{(i)}).
\end{align*}
\end{proof}

\begin{lemma} 
$Z_{\ubal}$ is irreducible.
\end{lemma}
\begin{proof}
We have already shown that $Z_{\ubal}$ is of pure codimension $1$.
We can stratify $V \setminus Z_{\ns}$ by the splitting type $\vec{a}$ of $\cK$ and the splitting type $\vec{f}$ of $\cE$. Each of these strata is isomorphic to $[U_{\vec{a}, \vec{f}} / \Aut(\cO(\vec{a})) \times \Aut(\cO(\vec{f}))]$, where $U_{\vec{a}, \vec{f}} \subset \Hom(\cO(\vec{a}), \cO(\vec{f}))$ is the locus of maps with locally free cokernel. Note that each stratum is irreducible. Note also that $\pi_* \cHom(\cO(\pvec{e}^{(i)}), \cO(D))$ is irreducible of codimension $1$ in $V$, and over this space the kernel is generically balanced. So when $\vec{a}$ is not balanced and $\vec{f} = \pvec{e}^{(i)}$ for some $i$, the stratum associated to $\vec{a}$ and $\vec{f}$ must be of codimension at least $2$ in $V$. In particular, it suffices to restrict our attention to the case where $\vec{f} = \vec{e}$, and check that the only unbalanced $\vec{a}$ such that $[U_{\vec{a}, \vec{e}} / \Aut(\cO(\vec{a})) \times \Aut(\cO(\vec{e}))]$ achieves codimension $1$ in $V \setminus Z_{\ns}$ is $\vec{a} = (-M-1, -M, \dots, -M, -M+1)$. 

Note that in order for $(-M, \dots, -M, D) < (e_1, \dots, e_r)$, we must have that $D \geq e_r + \sum_{i = 2}^{r-1}(e_i-e_1)$. So if $U_{\vec{a}, \vec{e}} \neq \emptyset$, then $\deg \vec{a} \leq (r-1)e_1$. In the special case where $\vec{e}$ itself is $(f_1^{r-1}, f_2)$, we must have that $D > f_2$. So if $U_{\vec{a}, \vec{e}} \neq \emptyset$ in this case, then $\deg \vec{a} < (r-1)e_1$.

We have 
\[
	\dim [U_{\vec{a}, \vec{e}} / \Aut(\cO(\vec{a})) \times \Aut(\cO(\vec{e}))] = \hom(\cO(\vec{a}), \cO(\vec{e})) - h^0\cEnd(\cO(\vec{a})) - h^0\cEnd(\cO(\vec{e})),
\]
and the dimension of $V \setminus Z_{\ns}$ is precisely 
\[
	\chi(\cHom(\cO(\vec{a}), \cO(\vec{e}))) - \chi(\cEnd(\cO(\vec{a}))) - h^0\cEnd(\cO(\vec{e})).
\]
So the codimension of $[U_{\vec{a}, \vec{e}} / \Aut(\cO(\vec{a}))]$ is $u(\vec{a}) - h^1 \cHom(\cO(\vec{a}), \cO(\vec{e}))$.

As explained by Hong and Larson in \cite[Theorem A.1.1]{HongLarson}, in order for there to exist an injective map $\cO(\vec{a}) \to \cO(\vec{e})$ with locally free cokernel, we could have the following two scenarios:
\begin{enumerate}
	\item $a_i \leq e_i$ for $1 \leq i \leq r-1$; or
	\item there exists some $1 \leq m \leq r$ such that $a_k = e_{k+1}$ for $k \geq m$, and $a_k < e_k$ for $k < m$.
\end{enumerate}

In the first case,
\begin{align*}
u(\vec{a}) - h^1 \cHom(\cO(\vec{a}), \cO(\vec{e})) &=
\sum_{i = 1}^{r-1}\sum_{j = 1}^{i-1}
\left (
h^1 \cHom(\cO(a_i), \cO(a_j)) - 
h^1 \cHom(\cO(a_i), \cO(e_j)) 
\right )
\end{align*}
where each term 
\[
	\delta_{ij} = 
	h^1 \cHom(\cO(a_i), \cO(a_j)) - 
	h^1 \cHom(\cO(a_i), \cO(e_j)) 
\]
is nonnegative because $a_j \leq e_j$. If all $\delta_{ij} \leq 1$ and $\delta_{ij} = 1$ exactly once, then whenever $a_i \geq a_j + 3$, we have $e_j - a_j \leq 1$ and $e_j = a_j + 1$ exactly once. Moreover, when $e_j = a_j + 1$, there can be at most one $i$ with $a_i-a_j \geq 2$.

Now suppose that $a_{r-1} - a_1 \geq 3$. Then we must have that $e_1 - a_1 \leq 1$. We cannot have $a_1 = e_1$, because then $\deg \vec{a} \geq (r-1)e_1 + 3$ would be too large. If $a_1 = e_1-1$, as we said before then it must be the case that $a_i - a_1 \leq 1$ for $i < r-1$. But then $a_{r-1} - a_i \geq 2$ for $1 < i < r-1$, which implies that $e_i = a_i$ for $1 < i < r-1$ (we've used up our one chance to have $e_i = a_i + 1$). This is now again impossible because $\deg \vec{a}$ is now too large.

Now suppose that $a_{r-1} - a_1 =2$. Then again because $\deg \vec{a} \leq (r-1)e_1$, we need that $e_1 - a_1 \geq 1$. But this implies that $a_i - a_1 \leq 1$ for $1 < i < r-1$. Thus up to a shift, $\vec{a}$ is of the form $((-1)^{k}, 0^s, 1)$, where $k > 0, s \geq 0$, $k+s = r-1$, and $e_1 \geq 0$. In particular, $h^1\cHom(\cO(\vec{a}), \cO(\vec{e})) = 0$. So $u(\vec{a}) = 1$, which implies that $k = 1$ and up to a shift, we have $\vec{a} = (-1, 0^{r-2}, 1)$, which is the expected splitting type.

Lastly, we will show that in the second case, the stratum $[U_{\vec{a}, \vec{e}}/\Aut(\cO(\vec{a})) \times \Aut(\cO(\vec{e}))]$ cannot be of codimension $1$. 
Let us write $\vec{e} = (\vec{e}_-, \vec{e}_+)$, $\vec{a} = (\vec{a}_-, \vec{e}_+)$, where $\vec{e}_+ = (e_{m+1}, e_{m+2}, \dots, e_{r}) = (a_{m}, \dots, a_{r-1})$. Note that $\vec{e}_-$ has length $m$ whereas $\vec{a}_-$ has length $m-1$.
\begin{align*}
& u(\vec{a}) - h^1 \cHom(\cO(\vec{a}), \cO(\vec{e})) \\ = \ & 
\left (
u(\vec{a}_-) - 
h^1 \cHom(\cO(\vec{a}_-), \cO(\vec{e}_-)
\right ) + 
\left(
	\ext^1(\cO(\vec{e}_+), \cO(\vec{a}_-)) -
	\ext^1(\cO(\vec{e}_+), \cO(\vec{e}_-))
\right).
\end{align*}
The first part is nonnegative for the same reason as in the first case. The second part can be computed directly as follows:
\begin{align*}
	& \ext^1(\cO(\vec{e}_+), \cO(\vec{a}_-)) -
	\ext^1(\cO(\vec{e}_+), \cO(\vec{e}_-)) \\ 
	= \ & 
	- \chi(\cHom(\cO(\vec{e}_+), \cO(\vec{a}_-))) + \chi(\cHom(\cO(\vec{e}_+), \cO(\vec{e}_-))) \\ 
	= \ & 
	- \left(
		(r-m)(m-1) + (r-m) \deg \vec{a}_- - (m-1) \deg(\vec{e}_+)
	\right)
	+ 
	\left(
		(r-m)m + (r-m) \deg \vec{e}_- - m \deg(\vec{e}_+)
	\right)	 \\ 
	= \ & (r-m)(1+D) - \deg \vec{e}_+. 
\end{align*}
But since $1+D > e_r$, the quantity $(r-m)(1+D) - \deg \vec{e}_+$ must be positive and it equals $1$ if and only if $m = r-1$ and $e_r = D$. But this is only possible if $-M = e_1 = e_2 = \cdots = e_{r-1}$, which just means that $\pvec{e}' = \vec{e}$, which is a contradiction.
\end{proof}

If we let $\alpha_1 = \sum_{i=1}^\ell \left (-(f_i+M)a_1^{(i)} + b_2^{(i)} \right ) $, then we see that
\[
	\twocase{\alpha_1 = }{[Z_{\ns}]}{\text{if $r = 2$},}{[Z_{\ubal}]}{\text{if $r \geq 3$}.}
\]
As a result of the previous lemmas, $\alpha_1$ is the only class that gets excised prior to the last step of framing the kernel bundle.

Next, by framing the universal kernel $\cK$, we introduce another relation $\alpha_2$ in $A^1$.
\begin{lemma} We have that
\[
	\Z \to A^1(U^{(1)}_{\pvec{e}'}) \to A^1(U^{(0)}_{\pvec{e}'}) \to 0,
\]
where the first map sends $1$ to
\[
	 \alpha_2 = c_1(\pi_*(\cK(M))) = \sum_{i = 1}^\ell \left ((f_i+M+1)a_1^{(i)} -b_2^{(i)} \right ).
\]
\end{lemma}
\begin{proof} 
Recall that forming $\Isom(\cK, \cO(-M)^{\oplus (r-1)})$ is the same as first taking $\pi_* \Hom(\cK, \cO(-M)^{\oplus (r-1)})$, and then taking the open corresponding to isomorphisms. At this stage of the construction, $\cK$ is fiberwise isomorphic to $\cO(-M)^{\oplus (r-1)}$, so $\cK \cong \pi^* (\pi_* \cK(M))(-M)$. The closed locus which one excises is cut out by the determinant, which can be understood as a section of $\wedge^{r-1} \cK^{\vee} \otimes \wedge^{r-1} \cO(-M)^{\oplus (r-1)} = \pi^* \wedge^{r-1}(\pi_* \cK(M))^{\vee}$, so the locus in $\pi_* \Hom(\cK, \cO(-M)^{\oplus (r-1)})$ corresponding to maps which are not isomorphisms has class $-c_1(\pi_* \cK(M))$. As in our computation for $[Z_{\ubal}]$, we find that
\begin{align*}
c_1(\pi_* \cK(M)) &= c_1(\pi_* \cE(M)) \\ 
&= \sum_{i = 1}^\ell c_1(\pi_* \cN_i(M)) \\ 
&= \sum_{i = 1}^\ell \left ((f_i+M+1)a_1^{(i)} -b_2^{(i)} \right ).
\end{align*}
\end{proof}

\subsection{The Class of the Canonical Module}
We give a compact description of the normal bundle for the smooth closed substack $\bZ_{\vec{e}}$ in $\cB_{r,d}^{\circ} = \cB_{r,d} \setminus \cup_{\vec{f} < \vec{e}, \vec{f} \neq \pvec{e}^{(i)}} \mathbf{\Sigma}_{\vec{f}}$ using filtered $\Ext$ and relative analogs. We refer to \cite[Section V.2.2.2]{Illusie} for more details on the functors $R\Hom_{\pm}(-, -)$, $R\cHom_{\pm}(-, -)$.
\begin{proposition} \label{prop:normalBundle}
The normal bundle of $\bZ_{\vec{e}}$ in $\cB_{r,d}^{\circ} = \cB_{r,d} \setminus \cup_{\vec{f} < \vec{e}, \vec{f} \neq \pvec{e}^{(i)}} \mathbf{\Sigma}_{\vec{f}}$ is $R^1 \pi_* \cEnd_+(\cE)$, which is filtered with associated graded isomorphic to $\oplus_{1 \leq j < i \leq \ell} R^1 \pi_* \cHom (\cN_i, \cN_j)$. In particular, for all smooth maps $U \to \cB_{r,d}^{\circ}$, there is an exact sequence
\[
	0 \to \cT_{\bZ_{\vec{e}} \times_{\cB_{r,d}^{\circ}} U} \to \cT_U \to R^1 \pi_* \cEnd_+(\cE) \to 0,
\]
where $\cT_{(-)}$ denotes the tangent bundle of the specified space.
\end{proposition}
\begin{proof}
Let $\rho: \SplittNye \to \Splitt$ be the resolution of singularities recalled at the beginning of Section \ref{sub:codimension_one_chow_group_of_splitt}. Thus, the first map in 
\[
	0 \to \cT_{\bZ_{\vec{e}}} \to \cT_{\cB_{r,d}^{\circ}}|_{\bZ_{\vec{e}}} \to \cN_{\bZ_{\vec{e}} \subset \cB_{r,d}^{\circ}} \to 0
\]
can be naturally identified with the natural map $\cT_{\SplittNye} \to \rho^* \cT_{\cB_{r,d}^{\circ}}$, restricted to $\rho^{-1}\bZ_{\vec{e}}$. Let us consider this situation after the base change along a smooth map $B \to \cB_{r,d}$, where $B$ is a scheme. Let $\cE$ be the vector bundle on $B \times \bP^1$ inducing the map $B \to \cB_{r,d}$, and let $Z_{\vec{e}} \subset U \subset B$ be the base change of the stacks $\mathbf{Z}_{\vec{e}} \subset \cB_{r,d}^{\circ} \subset \cB_{r,d}$. We would like to show that the normal bundle of $Z_{\vec{e}}$ in $U$ is $R^1 \pi_* \cEnd_+ \cE$, where $\cEnd_+(\cE)$ is taken with respect to the filtration on $\cE$ induced by 
$\SplittNye$.

Recall that the Kodaira-Spencer map $\cT_B \to R^1 \pi_* \cEnd(\cE)$ exists globally. Since $B \to \cB_{r,d}$ is smooth, this map is surjective. Moreover, the natural surjection $\cEnd(\cE) \to \cEnd_+(\cE)$ induces a surjection $R^1 \pi_* \cEnd(\cE) \to R^1 \pi_* \cEnd_+(\cE)$. The composition of these two surjections gives rise to a surjective map $\cT_B \to R^1 \pi_* \cEnd_+ \cE$. We claim that the composition 
\begin{equation} \label{eq: conormalExact}
	0 \to \cT_{Z_{\vec{e}}} \to \cT_U|_{Z_{\vec{e}}} \to R^1 \pi_* \cEnd_+ \cE \to 0
\end{equation}
is exact in the middle. 

This follows from Proposition 1.5 in \cite{Drezet-LePotier}. The content of their proposition is the following. Suppose that $\cE$ is a vector bundle on $B \times X$ where $X$ is a curve and the Kodaira-Spencer map $T_b B \to \Ext^1_X(\cE_b, \cE_b)$ is surjective for all $b \in B$. Then whenever we form a flag Quot construction associated to such $\cE$, then there is in general an exact sequence
\begin{equation} \label{eq:drezet-le-potier-seq}
	0 \to \Ext^0_+(\cE_b, \cE_b) \to T_t \FQuot(\cE) \to T_b B \xrightarrow{\omega_+} \Ext^1_+(\cE_b, \cE_b) \to 0,
\end{equation}
where $t \in \FQuot(\cE)$ is a point mapping to $b \in B$, corresponding to the data of a filtration on $\cE_b$. The $\Ext^i_+(\cE_b, \cE_b)$'s in the sequence are taken with respect to the filtration given by $t$, and $\omega_+$ is the composition of the Kodaira-Spencer map $T_b B \to \Ext^1(\cE_b, \cE_b)$ and the natural map $\Ext^1(\cE_b, \cE_b) \to \Ext^1_+(\cE_b, \cE_b)$. 

Applied to our situation, we can take the flag Quot construction which corresponds to base-changing $\SplittNye$ to $B$, which is an isomorphism over $Z_{\vec{e}}$. We know that we must have $\Ext^0_+(\cE_b, \cE_b) = 0$ for all $b \in Z_{\vec{e}}$ since $\SplittNye \to \Splitt$ is an isomorphism over $\bZ_{\vec{e}}$. This fact can be seen independently by using a spectral sequence to compute $\Ext^0_+$. The associated graded pieces arising from the filtration on $\cE$ are just the $\cN_i$'s, which are vector bundles. This implies that over $Z_{\vec{e}}$, $R\cHom_+(\cE, \cE) \cong \cHom_+(\cE, \cE)$. Since $R\Hom_+ = R\Gamma \circ R\cHom_+$ (\cite{Illusie} V.2.2.10), this shows that $(R^1 \pi_* \cEnd_+ \cE)_b \cong H^1(\bP^1, \cEnd_+(\cE_b, \cE_b)) \cong \Ext^1_+(\cE_b, \cE_b)$. Thus the exactness of \eqref{eq:drezet-le-potier-seq} shows that the base change of \eqref{eq: conormalExact} to any closed point $b \in B$ is exact in the middle.

Lastly, since $\cEnd_+(\cE)$ is filtered with associated graded isomorphic to $\oplus_{1\leq j < i \leq \ell} \cHom(\cN_i, \cN_j)$, and since $\pi_* \cHom(\cN_i, \cN_j) = 0$ for all $1\leq j < i \leq \ell$, the conormal bundle $R^1 \pi_* \cEnd_+(\cE)$ is filtered with associated graded isomorphic to $\oplus_{1 \leq j < i \leq \ell} R^1 \pi_* \cHom (\cN_i, \cN_j)$, as claimed.
\end{proof}

Even though $\Splitt$ is not smooth in general, if we only care about the class of the canonical module, we may excise codimension $\geq 2$ closed strata and work over a smooth open, where we have an explicit understanding of the normal bundle as given by the proposition above.

\begin{theorem} \label{prop:classCanonical} Let $X$ be a scheme, and let $\cE$ be a rank $r$, fiberwise degree $d$ vector bundle on $X \times \bP^1$, inducing a smooth map $X \to \cB_{r,d}$ (so that $X$ is necessarily smooth). Let $\vec{e} = (f_1^{s_1}, \dots, f_\ell^{s_\ell})$, where $f_1 < \cdots < f_\ell$, and let $\splitt(X)$ denote the splitting locus corresponding to $\vec{e}$ in $X$. Write $\delta_i = \sum_{j < i} s_j - \sum_{j > i} s_j$. Then
\[
	[\omega_{\splitt(X)}] = [\omega_X|_{\splitt(X)}] + \sum_{1 \leq i \leq \ell}\left ( (\deg \vec{e} - rf_i + \delta_i)a_1^{(i)} + (r-s_i)b_2^{(i)} \right).
\] 
\end{theorem}
\begin{proof}
Let $X^{\circ} = X \setminus \cup_{\vec{f} < \vec{e}, \vec{f} \neq \pvec{e}^{(i)}} \Sigma_{\vec{f}}$ and let $Z_{\vec{e}} = \bZ_{\vec{e}} \times_{\cB_{r,d}^{\circ}} X^{\circ}$. 
Using Proposition \ref{prop:normalBundle}, we find that
\begin{align*}
[\omega_{\splitt}] &= [\omega_{Z_{\vec{e}}}] \\ 
&= -c_1(\cT_{Z_{\vec{e}}}) \\ 
&= -c_1(\cT_{X^{\circ}}|_{Z_{\vec{e}}}) + c_1(\cN_{Z_{\vec{e}} \subset X^{\circ}}) \\ 
&= 
[\omega_X|_{\splitt}] + 
\sum_{1 \leq j < i \leq \ell} c_1 (R^1 \pi_*\cHom(\cN_i, \cN_j)).
\end{align*}

We claim that when $1 \leq j < i \leq \ell$,
\[
	c_1(R^1 \pi_*\cHom(\cN_i, \cN_j)) = (f_i-f_j-1)(s_i a_1^{(j)} - s_j a_1^{(i)}) + s_i b_2^{(j)} + s_j b_2^{(i)}, 
\]
which would yield the desired result.

First, we can consider the class of this vector bundle after we excise $\Sigma_{\pvec{e}^{(i)}}$. Then $\cN_i$ is perfectly balanced, and we have $\cN_i = \pi^* \cM_i(f_i)$, where we define $\cM_i = \pi_* \cN_i(-f_i)$. 

Now,
\begin{align*}
R^1 \pi_* \cHom(\cN_i, \cN_j) &= 
R^1 \pi_* \cHom(\pi^* \cM_i (f_i), \cN_j) \\ 
&= \cM_i^{\vee} \otimes R^1\pi_* \cN_j(-f_i).
\end{align*}
By Lemma \ref{lem:GRRCalc}, we have $c_1(R^1 \pi_* \cN_j(-f_i)) = -(f_j-f_i+1) a_1^{(j)} + b_2^{(j)}$, and $c_1(\cM_i^{\vee}) = -a_1^{(i)}$.
So if we excise $\Sigma_{\pvec{e}^{(i)}}$, then we see that modulo $b_2^{(i)}$,
\begin{align*}
	c_1 R^1 \pi_* \cHom(\cN_i, \cN_j) & = s_i c_1(R^1 \pi_* \cN_j(-f_i)) 
	+ s_j(f_i-f_j-1) c_1(\cM_i^{\vee}) \\ 
	& = 
	-s_i(f_j-f_i+1) a_1^{(j)} - s_j(f_i-f_j-1)a_1^{(i)} + s_ib_2^{(j)} \\ 
	&= (f_i-f_j-1)(s_i a_1^{(j)}-s_j a_1^{(i)}) + s_ib_2^{(j)}.
\end{align*}

Similarly, if we excise $\Sigma_{\pvec{e}^{(j)}}$, then 
\[
	R^1 \pi_* \cHom(\cN_i, \cN_j) = \cM_j \otimes R^1 \pi_* (\cN_i^{\vee}(f_j)) = \cM_j \otimes \pi_* (\cN_i(-f_j-2))^{\vee}.
\]
Again by Lemma \ref{lem:GRRCalc}, $c_1(\cM_j) = a_1^{(j)}$ and $c_1(\pi_* (\cN_i(-f_j-2))^{\vee}) = -(f_i-f_j-1)a_1^{(i)} + b_2^{(i)}$. So modulo $b_2^{(j)}$,
\[
	c_1 R^1 \pi_* \cHom(\cN_i, \cN_j)  
	= (f_i-f_j-1)(s_i a_1^{(j)}-s_j a_1^{(i)}) + s_jb_2^{(i)}.
\]
Since $b_2^{(i)}$ and $b_2^{(j)}$ are part of a free basis, the two calculations above together shows the desired formula.
\end{proof}

\subsection{Proof of Theorem \ref{thm:main}}
By Theorem \ref{prop:classCanonical}, the class of the canonical module for $\splitt \subset H^1 \cEnd(\cO(\pvec{e}'))$ is 
\begin{align*}
[\omega_{\splitt}] = \sum_{1 \leq i \leq \ell}\left ( (\deg \vec{e} - rf_i + \delta_i)a_1^{(i)} + (r-s_i)b_2^{(i)} \right).
\end{align*}
Recall also that the two classes that generate the kernel of $A^1(\Splitt) \to A^1(\splitt)$ are 
\begin{align*}
	\alpha_1 &= \sum_{i=1}^\ell \left (-(f_i+M)a_1^{(i)} + b_2^{(i)} \right ), \\ 
	\alpha_2 &= \sum_{i = 1}^\ell \left ((f_i+M+1)a_1^{(i)} -b_2^{(i)} \right ).
\end{align*}
For convenience, given two vectors $\vec{p}, \vec{q}$ both of length $\ell$, we use the notation $(\vec{p} \mid \vec{q})$ to denote the element $\sum_{i = 1}^\ell (p_i a_1^{(i)} + q_i b_2^{(i)})$.
The subgroup generated by $\alpha_1, \alpha_2$ is free of rank $2$, which we think of in terms of the following basis:
\begin{align*}
	(\mathbf{1}_{\vec{a}} \mid \mathbf{0} ) &= \sum_{i=1}^\ell a_1^{(i)}, \\ 
	(\vec{f} \mid -\mathbf{1}_{\vec{b}}) &= \sum_{i=1}^\ell \left ( f_ia_1^{(i)} - b_2^{(i)} \right ).
\end{align*}
We can rewrite the formula for $[\omega_{\splitt}]$ from Theorem \ref{prop:classCanonical} in terms of this basis:
\[
	[\omega_{\splitt}] = \deg \vec{e} (\mathbf{1}_{\vec{a}} \mid \mathbf{0} ) - r (\vec{f} \mid -\mathbf{1}_{\vec{b}}) + (\vec{\delta} \mid  -\vec{s}),
\]
where $\vec{\delta} = (\delta_1, \dots, \delta_{\ell})$ and $\vec{s} = (s_1, \dots, s_{\ell})$. So $[\omega_{\splitt}]$ is zero or $N$-torsion if and only if $(\vec{\delta} \mid  -\vec{s})$ is zero or $N$-torsion. We now explain exactly when the canonical class is trivial or torsion. 

Let $\vec{\delta} = (\delta_1, \dots, \delta_{\ell})$. Let $\Delta(\vec{\delta}) = (\delta_2 - \delta_1, \dots, \delta_{\ell} - \delta_{\ell-1}) = (s_1+s_2, \dots, s_{\ell-1} + s_\ell)$, and similarly let $\Delta(\vec{f}) = (f_2 - f_1, \dots, f_\ell - f_{\ell-1})$. Let $\pvec{e}^{(i_1)}, \dots, \pvec{e}^{(i_k)}$ be those splitting types which are codimension one in $\Splitt$. 

First suppose there is at least one such codimension one splitting type. Then by observing the generators, we see that if $N(\vec{\delta} \mid -\vec{s}) \in \operatorname{span} \langle \alpha_1, \alpha_2 \rangle$ for some $N \geq 1$, then we must have that $s_{i_1} = s_{i_2} = \cdots = s_{i_k} = s$ for some $s \geq 2$, and there exist some $A, B$ such that
\[
	(N\vec{\delta} \mid -Ns\mathbf{1}_{\vec{b}}) = (N\vec{\delta} \mid -N\vec{s}) = A(\mathbf{1}_{\vec{a}} \mid \mathbf{0}) + B(\vec{f} \mid -\mathbf{1}_{\vec{b}}),
\]
But then we are forced to take $B = Ns$. Looking only at the $\vec{a}$ part of the basis, we see that
\[
	N\vec{\delta} = A \cdot \mathbf{1}_{\vec{a}} + Ns \vec{f}. 
\]
Equivalently, $N\Delta(\vec{\delta}) = Ns \Delta(\vec{f})$.
But if this is true, then we also have that $\Delta(\vec{\delta}) = s \Delta(\vec{f})$. Let $i \in \{i_1, \dots, i_k\}$ so that $s_i = s$. Then since $s_{i-1} + s_i = \Delta(\vec{\delta})_{i-1}$ is divisible by $s$, $s_{i-1}$ must also be divisble by $s$. Similarly, since $s_{i} + s_{i+1} = \Delta(\vec{\delta})_{i}$ is divisible by $s$, $s_{i+1}$ must be divisible by $s$. Hence we can write $s_{i} = k_i s$ for each $1 \leq i \leq \ell$, where $k_i > 0$. The condition tells us that $f_{i+1}-f_{i} = \Delta(\vec{f})_{i} = \frac{1}{s}\Delta(\vec{\delta})_{i} = k_i + k_{i+1}  \geq 2$. Moreover, $s_i = k_i s \geq s \geq 2$. So actually every $1 \leq i \leq \ell$ contributes a codimension $1$ class. It follows that $s_i = s$ for all $1 \leq i \leq \ell$, and $f_{i+1}-f_i = 2$ for all $1 \leq i < \ell$. In sum, we have shown that if there exists a codimension one splitting type $\pvec{e}^{(i)}$ and $\Splitt$ is $N$-Gorenstein, then $\vec{e}$ is a block arithmetic progression of difference $2$ and $\Splitt$ is actually Gorenstein.

Now suppose that there are no codimension $1$ splitting loci in $\Splitt$. In this case, $\Splitt$ is $N$-Gorenstein if and only if $N \Delta(\vec{\delta}) = B \Delta(\vec{f})$ for some $B \geq 1$. For each $1 \leq i \leq \ell$, there are two possibilities:
\begin{enumerate}
	\item $f_{i+1} = f_i+1$ or $f_{i-1} = f_i-1$, and $s_i$ can be anything;
	\item $f_i \geq f_{i-1} + 2$ and $f_{i} \leq f_{i+1}-2$, and $s_i = 1$.
\end{enumerate}

When only the first possibility occurs, we claim that in fact for every $1 \leq i < \ell$, $f_{i+1} = f_{i} + 1$. This is trivial if $\ell \leq 3$, so let us assume that $\ell \geq 4$. Suppose towards the contrary that $f_{i+1} - f_i \geq 2$ for some $i$, which can only happen for $2 \leq i \leq \ell-2$. Then we must have $f_{i-1} = f_i-1$ and $f_{i+2} = f_{i+1}+1$. So the relevant entries of $\Delta(\vec{f})$ and $\Delta(\vec{\delta})$ are $(\dots, 1, f_{i+1}-f_i, 1, \dots)$ and $(\dots, s_{i-1}+s_i, s_i+s_{i+1}, s_{i+1}+s_{i+2}, \dots)$. Since $\Delta(\vec{f})$ and $\Delta(\vec{\delta})$ are proportional, we must have
\begin{align*}
s_{i-1}+s_i &= s_{i+1}+s_{i+2} \\ 
s_i+s_{i+1} &= (f_{i+1}-f_i)(s_{i-1}+s_i).
\end{align*}
In particular,
\[
	2(s_i + s_{i+1}) = (f_{i+1}-f_i)(s_{i-1} + s_i + s_{i+1} + s_{i+2}),
\]
which is impossible. Thus we conclude that $\Delta(\vec{f}) = (1, 1, \dots, 1)$, and there exist $N, B$ such that $N\Delta(\vec{\delta}) = B \Delta(\vec{f})$ if and only if there exists $B \geq 1$ such that $\Delta(\vec{\delta}) = B \Delta(\vec{f})$. The latter condition is saying that $s_{i-1} + s_i = s_{i} + s_{i+1}$ for all $1 < i < \ell$. This is exactly the case where $\vec{e}$ is contiguous. In sum, if a splitting type $\vec{e}$ is $N$-Gorenstein, does not dominate any splitting type of codimension $1$, and only possibility (1) occurs, then $\Splitt$ is contiguous and it is Gorenstein.

If only the second possibility occurs, we have $\Delta(\vec{\delta}) = (2, 2, \dots, 2)$. So in this case if $\Splitt$ is $N$-Gorenstein for any $N \geq 1$, then $\vec{e}$ is an arithmetic progression. If $\vec{e}$ is an arithmetic progression with difference $t$, then $\Splitt$ is Gorenstein if and only if $t = 0,1,2$. When $t > 2$, if $t = 2t'$, then the smallest $N$ such that $\Splitt$ is $N$-Gorenstein is $N = t'$. If $t$ is odd, then the smallest $N$ such that $\Splitt$ is $N$-Gorenstein is $t$.

Lastly, we claim that the two possibilities cannot both occur for an $N$-Gorenstein splitting type $\vec{e}$ with no codimension $1$ strata. It suffices to show that it is not possible to have $i, j$ such that $|i-j| = 1$, (1) occurs at $i$ and (2) occurs at $j$. We make the argument for $j = i-1$, as the other case is entirely analogous. In this case, $\vec{e} = (\dots, f_{i-1}, f_i^{s_i}, f_{i+1}^{s_{i+1}}, \dots)$, where $f_{i+1} = f_i+1$. But then $\Delta(\vec{f})_{i-1} = f_i-f_{i-1} > 1 = \Delta(\vec{f})_i$. This implies that $1+s_i = \Delta(\vec{\delta})_{i-1} > \Delta(\vec{\delta})_i = s_i + s_{i+1}$, which is not possible. 

\printbibliography

\end{document}